\documentclass[11pt,a4paper,leqno]{amsart}
\usepackage[utf8]{inputenc}
\usepackage{amsmath,amssymb,amsfonts,amsthm,yfonts}
\usepackage{latexsym,epsfig,graphicx,subfigure}
\usepackage{mathtools}
\usepackage{esint}
\usepackage{bbm}
\usepackage{calrsfs}
\usepackage{esint}
\usepackage{mathrsfs}
\usepackage{float}
\DeclareFontFamily{U}{mathx}{}
\DeclareFontShape{U}{mathx}{m}{n}{<-> mathx10}{}
\DeclareSymbolFont{mathx}{U}{mathx}{m}{n}
\DeclareMathAccent{\widehat}{0}{mathx}{"70}
\DeclareMathAccent{\widecheck}{0}{mathx}{"71}
\usepackage{graphicx,epsfig,subfigure}
\usepackage{enumitem}
\usepackage{latexsym}
\usepackage{hyperref}
\usepackage{etoolbox}
\usepackage{tikz}
\usetikzlibrary{decorations.markings,calc}
\makeatletter
\def\thm@space@setup{%
	\thm@preskip=8pt
	\thm@postskip=8pt
}
\makeatother
\AtEndEnvironment{proof}{\vspace{6pt}}

\hypersetup{
	colorlinks=true,
	linkcolor=magenta,
	filecolor=black,
	urlcolor=blue,
	citecolor=cyan,
}

\newtheorem{theorem}{Theorem}
\newtheorem{proposition}[theorem]{Proposition}
\newtheorem{lemma}[theorem]{Lemma}
\newtheorem{corollary}[theorem]{Corollary}
\theoremstyle{definition}
\newtheorem{definition}[theorem]{Definition}
\theoremstyle{remark}
\newtheorem{remark}[theorem]{Remark}

\numberwithin{equation}{section}

\newcommand{\R}{\mathbb{R}}
\newcommand{\C}{\mathbb{C}}
\newcommand{\dist}{\operatorname{dist}}
\newcommand{\pv}{\operatorname{p.v.}}
\newcommand{\Cc}{C_c^\infty(\R)}
\newcommand{\norm}[1]{\left\|#1\right\|}
\newcommand{\ip}[2]{\langle #1,#2\rangle}

\DeclareMathAlphabet{\pazocal}{OMS}{zplm}{m}{n}

\makeatletter
\newcommand*{\rom}[1]{\expandafter\@slowromancap\romannumeral #1@}
\makeatother

\begin{document}
	
	\title{Uniform \(L^2\)-bounds for the \(n\)-th Calderón commutator on Lipschitz graphs}
	\author{Joan Hernández, Joan Mateu and Laura Prat}\thanks{Joan Hern\'andez and Laura Prat have been supported by PID2024-165107NB-I00 and Joan Mateu by PID2024-155320NB-I00 (Ministerio de Ciencia e Innovación, Spain).}

	\begin{abstract}
		We prove that, for every real Lipschitz function $A$ and every integer $n\ge1$,
		$\displaystyle
		\norm{T_{A,n}}_{L^2(\R)\to L^2(\R)}\le C\,\norm{A'}_\infty^n ,$
		where $T_{A,n}$ is the $n$-th Calderón commutator and $C$ is an absolute constant. 
		\bigskip
		
		
		\medskip
		
	\end{abstract}
\maketitle	
  \section{Introduction}
	
	Calderón commutators are closely related to the Cauchy singular integral on Lipschitz graphs. Given a real
	Lipschitz function $A:\R\to\R$ and an integer $n\geq 1$, the
	$n$-th Calder\'on commutator is defined by
	\begin{equation}\label{eq:calderon}
		T_{A,n}f(x)
		=\pv\int_\R
		\left(\frac{A(x)-A(y)}{x-y}\right)^n
		\frac{f(y)}{x-y}\,dy.
	\end{equation}
	A fundamental question is how the $L^2$ norm of $T_{A,n}$
	depends on the order $n$ and on the Lipschitz constant
	$\norm{A'}_\infty$.
	
	The connection with the Cauchy integral becomes apparent
	when one considers the graph
	\[
	\Gamma_A=\{x+iA(x):x\in\R\}.
	\]
	For $\norm{A'}_\infty<1$, the Cauchy kernel admits the expansion
	\[
	\frac{1}{x-y+i(A(x)-A(y))}
	=\sum_{n=0}^\infty(-i)^n
	\frac{(A(x)-A(y))^n}{(x-y)^{n+1}},
	\]
	so that the Calder\'on commutators appear as the coefficients
	of a natural analytic family of singular integral operators.
	
	The $L^2$ boundedness of these operators has a long history.
	Calder\'on \cite{C1,C2} established the boundedness of the first
	commutator and obtained results for the Cauchy integral under
	a smallness assumption on the Lipschitz constant.
	The breakthrough of Coifman, McIntosh and Meyer \cite{CMM}
	removed this restriction and established the boundedness of
	the Cauchy integral on arbitrary Lipschitz graphs. Their
	methods also gave estimates for higher-order commutators
	with polynomial growth of order $(1+n)^4$.
	
	The dependence on $n$ was subsequently improved by
	Christ and Journ\'e \cite{ChJ}, who proved that, for every
	$\varepsilon>0$,
	\[
	\norm{T_{A,n}}_{L^2\to L^2}
	\leq C_\varepsilon(1+n)^{1+\varepsilon}
	\norm{A'}_\infty^n.
	\]
	A linear estimate, previously indicated by Verdera \cite{V} on the basis of joint work with Mateu, was later established in \cite{HMP}:
	\[
	\norm{T_{A,n}}_{L^2\to L^2}
	\leq Cn\norm{A'}_\infty^n.
	\]
	For functions $A$ supported in $[0,1]$, the same work obtained an improvement to $\sqrt{n}$ growth under additional regularity assumptions on $A'$, including Dini and logarithmic Besov conditions.
	
	These results naturally raise the question of whether
	any growth in $n$ is necessary at all. The main result
	of this paper gives a negative answer: the $L^2$ norms
	of the Calder\'on commutators admit a bound independent
	of their order, with no regularity assumptions beyond
	the Lipschitz continuity of $A$.
	
	\begin{theorem}[Main theorem]\label{thm:main}
		There exists an absolute constant $C$ such that,
		for every real Lipschitz function $A$ and every
		integer $n\geq 1$,
		\[
		\norm{T_{A,n}}_{L^2(\R)\to L^2(\R)}
		\leq C\norm{A'}_\infty^n.
		\]
	\end{theorem}
	Observe that if $L=\norm{A'}_\infty=0$, then $A$ is constant and $T_{A,n}=0$. If $L>0$, then $T_{A,n}=L^nT_{A/L,n}$. It thus suffices to treat the case $\norm{A'}_\infty\le1$, which we assume from now on. 
	
	The proof of Theorem~\ref{thm:main} uses the Cauchy integral. Rather than estimating each commutator separately, we
	recover the operators $T_{A,n}$ through Cauchy's integral
	formula applied to a family of Cauchy kernels associated
	with complex deformations of the underlying graph.
	The Coifman--McIntosh--Meyer theorem provides the required
	$L^2$ bounds for these kernels, provided the corresponding
	curves remain Lipschitz graphs with uniformly controlled
	constants.
	
	The essential point is therefore geometric. We construct
	contours in the complex plane along which the
	associated curves have uniformly bounded Lipschitz
	constants, while the contribution of the contour integral
	remains bounded independently of $n$. 
	
	The paper is organised as follows. Section~\ref{sec:prelim} collects the necessary preliminaries, including the Coifman--McIntosh--Meyer theorem, and concludes with an outline of the proof of Theorem~\ref{thm:main}. In Section~\ref{sec:geometry}, we establish a geometric criterion ensuring that a curve can be represented, after a suitable rotation, as a Lipschitz graph. In Section~\ref{sec:family}, we introduce and study a complex family of Cauchy kernels and obtain $L^2$ estimates for the
	associated bilinear forms. Section~\ref{sec:extraction} shows how to recover the Calder\'on commutators from this family using Cauchy's integral formula. In Section~\ref{sec:contour}, we first obtain an $O(\log n)$ estimate using a circular contour and
	then construct suitable contours that remove the logarithmic loss and complete the proof of
	Theorem~\ref{thm:main}. Finally, Section~\ref{sec:remarks} discusses sharpness and some further remarks.

	\section{Preliminaries}\label{sec:prelim}
	In this section, we collect the preliminary results needed for the proof of Theorem~\ref{thm:main} and conclude with an outline of the argument.
	\subsection{Basic results}
	We use the bilinear pairing $\ip fg=\int_\R f g$, without complex conjugation. If $T$ is a bounded linear operator on $L^2(\R)$, then we write
	\[
	\norm T=\sup\{|\ip{Tf}{g}|:\norm f_2=\norm g_2=1\},
	\]
	and it suffices to take $f,g$ in a dense subset. 
	
	
	For $f,g\in\Cc$, let $S$ be a compact interval containing $\operatorname{supp}(f)\cup\operatorname{supp}(g)$. Then $f(y)g(x)-f(x)g(y)$ vanishes outside $S\times S$. Moreover, writing $f(y)g(x)-f(x)g(y)=f(y)\bigl(g(x)-g(y)\bigr)-g(y)\bigl(f(x)-f(y)\bigr)$, we get
	\begin{equation}\label{eq:D}
		|f(y)g(x)-f(x)g(y)|\le c_{f,g}\,|x-y|,\qquad c_{f,g}:=\norm f_\infty\norm{g'}_\infty+\norm g_\infty\norm{f'}_\infty .
	\end{equation}

	\begin{lemma}\label{lem:weak}
		Let $K:\R^2\setminus\{x=y\}\to\C$ be a measurable
		antisymmetric kernel satisfying
		$\displaystyle
		|K(x,y)|\leq \frac{B}{|x-y|},$ $x\neq y.$
		Let $\rho:\R^2\to[0,\infty)$ be a measurable symmetric
		function such that $\rho(x,x)=0$ and
		$\displaystyle
		0<\rho(x,y)\leq L|x-y|,$ $x\neq y,
		$
		for some $L>0$. Then, for every $f,g\in\Cc$,
		\[
		\lim_{\varepsilon\to0^+}
		\underset{\rho(x,y)>\varepsilon}\iint
		K(x,y)f(y)g(x)\,dx\,dy
		=
		\frac12\underset{\R^2}\iint
		K(x,y)\bigl(f(y)g(x)-f(x)g(y)\bigr)\,dx\,dy.
		\]
		Moreover, the integral on the right is absolutely convergent.
	\end{lemma}
	
	\begin{proof}
		Fix $\varepsilon>0$. Since
		$\displaystyle
		\{\rho(x,y)>\varepsilon\}
		\subseteq
		\left\{|x-y|>\frac{\varepsilon}{L}\right\},
		$
		the truncated integral is absolutely convergent.
		
		By the symmetry of $\rho$ and the antisymmetry of $K$,
		exchanging $x$ and $y$ and averaging the two expressions, we obtain
		\[
		\iint_{\rho(x,y)>\varepsilon}
		K(x,y)f(y)g(x)\,dx\,dy
		=
		\frac12
		\iint_{\rho(x,y)>\varepsilon}
		K(x,y)(f(y)g(x)-f(x)g(y))\,dx\,dy.
		\]
		
		By \eqref{eq:D}, 
		$\displaystyle
		|K(x,y)(f(y)g(x)-f(x)g(y))|
		\leq Bc_{f,g}\mathbf 1_{S\times S}(x,y),
		$
		which is integrable on $\R^2$.
		
		Finally, since $\rho(x,y)>0$ whenever $x\neq y$,
		we have
		$\displaystyle
		\mathbf 1_{\{\rho(x,y)>\varepsilon\}}
		\longrightarrow 1
		$
		for almost every $(x,y)\in\R^2$ as
		$\varepsilon\to0^+$.
		The conclusion follows from the dominated
		convergence theorem.
	\end{proof}
	
	\subsection{The commutator as a bilinear form}\label{sec2}
	Assume $\norm{A'}_\infty\le1$ and let 
	\[
	q(x,y)=\frac{A(x)-A(y)}{x-y}, \quad \text{for} \quad x\ne y.
	\]
	Then $q$ is symmetric and takes values in $[-1,1]$. The kernel of \eqref{eq:calderon} is
	\[
	K_n(x,y)=\frac{q(x,y)^n}{x-y},
	\]
	which is antisymmetric and satisfies $|K_n(x,y)|\le|x-y|^{-1}$. Set 
	$$T_{A,n,\varepsilon}f(x)=\int_{|x-y|>\varepsilon}K_n(x,y)f(y)\,dy.$$
	Lemma~\ref{lem:weak}, applied with $\rho(x,y)=|x-y|$, shows that for $f,g\in \Cc$,
	\[
	\lim_{\varepsilon\to0^+}
	\ip{T_{A,n,\varepsilon}f}{g}=\frac12\iint_{\R^2}K_n(x,y)\,(f(y)g(x)-f(x)g(y))\,dx\,dy, 
	\]
	where the integral on the right is  absolutely convergent. 
	
	Define
	\begin{equation}\label{eq:tn}
		t_n(f,g):=\frac12\iint_{\R^2}
		K_n(x,y)\,(f(y)g(x)-f(x)g(y))\,dx\,dy,
		\qquad f,g\in\Cc.
	\end{equation}
For each fixed $n$, the kernel $K_n$ is a standard
Calder\'on--Zygmund kernel, with constants possibly
depending on $n$. Indeed, the size estimate is immediate,
and the smoothness estimates follow from the Lipschitz
continuity of $A$.

By the classical $L^2$ boundedness of Calder\'on
commutators and standard Calder\'on--Zygmund theory,
the symmetric truncations are uniformly bounded on
$L^2(\R)$, for each fixed $n$, and their principal
values exist almost everywhere; see \cite{Jo,Da}.
Consequently, the truncations converge weakly in
$L^2(\R)$ to $T_{A,n}$.

Lemma~\ref{lem:weak} therefore gives
\[
\ip{T_{A,n}f}{g}=t_n(f,g),
\qquad f,g\in\Cc.
\]
We use these classical results only for this
identification; no estimate uniform in $n$ is needed
at this stage.

Consequently, by density, the desired uniform $L^2$
estimate is equivalent to proving
\[
|t_n(f,g)|\le C\norm f_2\norm g_2,
\qquad f,g\in\Cc,
\]
with $C$ independent of $n$.

	\subsection{The Coifman--McIntosh--Meyer theorem}

				A \emph{rotated Lipschitz graph} with constant $\le M$
		is a curve of the form
		\[
		\Gamma=e^{i\theta}\{s+i\varphi(s):s\in\R\},
		\]
		where $\theta\in\R$ and $\varphi:\R\to\R$ is Lipschitz
		with $\norm{\varphi'}_\infty\le M$.
		We denote by $\sigma$ the arclength measure on $\Gamma$.
		Thus, under the parametrization
		$z=e^{i\theta}(s+i\varphi(s))$,
		\[
		d\sigma(z)=\sqrt{1+\varphi'(s)^2}\,ds.
		\]
		
		For $\varepsilon>0$, define the truncated Cauchy operator
		\[
		\mathcal C_{\Gamma,\varepsilon}F(z)
		=
		\int_{\{w\in\Gamma:\,|z-w|>\varepsilon\}}
		\frac{F(w)}{z-w}\,d\sigma(w),
		\qquad z\in\Gamma.
		\]
		
		\begin{definition}\label{def:CM}
			For $M\ge0$, set
			\[
			C_M:=
			\sup\Bigl\{
			\norm{\mathcal C_{\Gamma,\varepsilon}}_{L^2(\sigma)\to L^2(\sigma)}:
			\Gamma \text{ a rotated Lipschitz graph with constant }\le M,
			\ \varepsilon>0
			\Bigr\}.
			\]
			By construction, $M\mapsto C_M$ is non-decreasing.
		\end{definition}
		
		\begin{theorem}[Coifman--McIntosh--Meyer \cite{CMM}]
			\label{thm:CMM}
			For every $M\ge0$, the constant $C_M$ defined above is finite.
		\end{theorem}
		
	The Coifman--McIntosh--Meyer theorem is originally formulated for Lipschitz graphs without rotations and with respect to the complex measure $dw$. We explain below why the formulation stated above, involving rotated Lipschitz graphs, arclength measure, and uniform bounds for the truncated operators, follows from their theorem and standard results in Calderón--Zygmund theory.
		
		\begin{enumerate}[label=(\alph*)]
			
	\item \emph{Arclength measure.}
	Let $\Gamma_0=\{s+i\varphi(s):s\in\R\}$, and denote by $\sigma_0$ its arclength measure and by $\tau(w)$ its unit tangent, defined almost everywhere. Since
		\[
		dw=(1+i\varphi'(s))\,ds=\tau(w)\,d\sigma_0(w),
		\qquad |\tau(w)|=1,
		\]
		we have
		\[
		\int_{\Gamma_0}\frac{F(w)}{z-w}\,d\sigma_0(w)
		=
		\int_{\Gamma_0}
		\frac{\overline{\tau(w)}F(w)}{z-w}\,dw.
		\]
		Thus, replacing $dw$ by $d\sigma_0$ amounts to
		composing the Cauchy operator with multiplication
		by $\overline{\tau}$, which is unitary on $L^2(\sigma_0)$.
		Consequently, the two operators have the same
		$L^2(\sigma_0)$ norm.

			\item \emph{Rotations.}
			Suppose that $\Gamma=e^{i\theta}\Gamma_0$, where
			$\Gamma_0$ is a Lipschitz graph, and let $\sigma_0$
			and $\sigma$ denote the corresponding arclength measures.
			The map
			\[
			U_\theta:L^2(\sigma)\longrightarrow L^2(\sigma_0),
			\qquad
			U_\theta F(z)=F(e^{i\theta}z),
			\]
			is unitary, since rotations preserve arclength measure.
			Moreover, rotations preserve distances, and a change
			of variables gives
			\[
			U_\theta\mathcal C_{\Gamma,\varepsilon}U_\theta^{-1}
			=
			e^{-i\theta}\mathcal C_{\Gamma_0,\varepsilon}.
			\]
			Therefore,
			\[
			\norm{\mathcal C_{\Gamma,\varepsilon}}_{L^2(\sigma)\to L^2(\sigma)}
			=
			\norm{\mathcal C_{\Gamma_0,\varepsilon}}_{L^2(\sigma_0)\to L^2(\sigma_0)}.
			\]
			Hence the $L^2$ estimates for Lipschitz graphs extend
			to rotated Lipschitz graphs with exactly the same constants.
			
			\item \emph{Uniform truncations.}
			The Coifman--McIntosh--Meyer theorem gives the
			$L^2$ boundedness of the principal-value Cauchy operator
			on Lipschitz graphs, with a bound depending only on
			the Lipschitz constant $M$.
			To obtain the uniform boundedness of the truncated
			operators, we use Cotlar's inequality.
			Indeed, the kernel $1/(z-w)$ restricted to a Lipschitz
			graph is a standard Calder\'on--Zygmund kernel with
			respect to arclength measure, with constants depending
			only on $M$.
			Cotlar's inequality therefore yields
			\[
			\sup_{\varepsilon>0}
			\norm{\mathcal C_{\Gamma,\varepsilon}}_{L^2(\sigma)\to L^2(\sigma)}
			\le C(M),
			\]
			where $C(M)$ is independent of the particular graph
			and of $\varepsilon$; see \cite{Jo,Da}.
			
		\end{enumerate}
	For our purposes, it suffices to know that $C_{2+\sqrt{5}}$ is finite. No quantitative information on the dependence of $C_M$ on $M$ is needed.
	
 \subsection*{Outline of the proof of Theorem \ref{thm:main}}

We now explain the main idea of the proof.
Recall that $t_n$ is the bilinear form associated with
the $n$-th Calder\'on commutator.

For $\lambda\in\C\setminus[-1,1]$, consider the curve
$\gamma_\lambda(x)=\lambda x-A(x)$ and its associated
Cauchy kernel
\[
K_\lambda(x,y)
=\frac{1}{\gamma_\lambda(x)-\gamma_\lambda(y)}
=\frac{1}{(x-y)(\lambda-q(x,y))}.
\]
Since $|q(x,y)|\le1$, for $|\lambda|>1$ we have
\[
K_\lambda(x,y)
=\frac{1}{\lambda(x-y)}
+\sum_{n=1}^\infty
\frac{K_n(x,y)}{\lambda^{n+1}}.
\]
Thus the kernels of the Calder\'on commutators appear
as coefficients in the Laurent expansion of \(K_\lambda\) at infinity. For $f,g\in\Cc$, define
the corresponding bilinear forms
\[
r_\lambda(f,g)
:=\frac12\iint_{\R^2}
K_\lambda(x,y)
\bigl(f(y)g(x)-f(x)g(y)\bigr)\,dx\,dy.
\]

The proof has two main ingredients.

\begin{enumerate}

\item \emph{Geometric control of the Cauchy kernels.}

Since $\gamma_\lambda'(x)=\lambda-A'(x)$ almost
everywhere, its tangent vectors lie in the segment
$[\lambda-1,\lambda+1]$.
We shall show that $\gamma_\lambda(\R)$ is, after
a rotation, a Lipschitz graph with constant at most
\[
M(\lambda):=\tan\frac{\omega(\lambda)}2,
\]
where $\omega(\lambda)$ is the angle under which
the segment $[-1,1]$ is seen from $\lambda$.
The Coifman--McIntosh--Meyer theorem then gives
\[
|r_\lambda(f,g)|
\le
\frac{C_{M(\lambda)}}{d_\lambda}
\norm f_2\norm g_2,
\]
where $d_\lambda=\dist(\lambda,[-1,1])$.

\item \emph{Recovery of the commutators.}

We shall prove that, for any positively oriented piecewise \(C^1\) Jordan contour \(\Gamma\) enclosing \([-1,1]\),
\[t_n(f,g)
=\frac{1}{2\pi i}
\int_\Gamma\lambda^n r_\lambda(f,g)\,d\lambda.
\]
This identity follows from Fubini's theorem and
Cauchy's integral formula applied to
$\lambda^n/(\lambda-q(x,y))$.
Unlike the series expansion above, it remains valid
for contours entering the unit disc.

Combining this representation with the preceding
estimate, we obtain
\[
|t_n(f,g)|
\le
\frac{\norm f_2\norm g_2}{2\pi}
\int_\Gamma
\frac{|\lambda|^n C_{M(\lambda)}}{d_\lambda}
\,|d\lambda|.
\]

\end{enumerate}

It remains to construct, for each $n$, a contour
along which $M(\lambda)$ is uniformly bounded and
\[
\int_\Gamma
\frac{|\lambda|^n}{d_\lambda}\,|d\lambda|
\]
is bounded independently of $n$.
A circular contour gives an $O(\log n)$ bound.
We shall remove this logarithmic loss by constructing
contours that enter the unit disc while keeping
the associated Lipschitz constants uniformly bounded.

	\section{Curves with tangent in a cone}\label{sec:geometry}
In this section we establish two geometric lemmas that will allow us to represent the curves \(\gamma_\lambda(\mathbb R)\) as rotated Lipschitz graphs and control their Lipschitz constants in terms of \(\lambda\).	

	For $\theta\in\R$ and $\Theta\in[0,\pi)$, the \emph{closed sector of opening $\Theta$ and bisector $\theta$} is the set
	\[
	\mathcal S(\theta,\Theta)=\{re^{i\psi}:\ r\ge0,\ |\psi-\theta|\le\Theta/2\}.
	\]
	
	\begin{lemma}\label{lem:cone}
		Let $\gamma:\R\to\C$ be Lipschitz and suppose that there are $\theta\in\R$, $\Theta\in[0,\pi)$ and $m>0$ such that, almost everywhere,
		\[
		\gamma'(x)\in\mathcal S(\theta,\Theta)\qquad\text{and}\qquad|\gamma'(x)|\ge m .
		\]
		Then:
		\begin{enumerate}[label=(\roman*)]
			\item[i\,\normalfont{)}]  $|\gamma(x)-\gamma(y)|\ge m\cos(\Theta/2)\,|x-y|$; in particular $\gamma$ is injective.
			\item[ii\,\normalfont{)}] $\Gamma=\gamma(\R)$ is a rotated Lipschitz graph with constant $\le\tan(\Theta/2)$.
			\item[iii\,\normalfont{)}] For every non-negative measurable function $F$ on $\Gamma$, $$\displaystyle\int_\Gamma F\,d\sigma=\int_\R F(\gamma(x))\,|\gamma'(x)|\,dx.$$
		\end{enumerate}
	\end{lemma}
	\begin{proof}
		Write $e^{-i\theta}\gamma=u+iv$. Then, almost everywhere,
		$\displaystyle
		u'\ge m\cos(\Theta/2)>0$ and $\displaystyle|v'|\le\tan(\Theta/2)\,u'$. Integrating the first inequality, $u(x)-u(y)\ge m\cos(\Theta/2)(x-y)$ for $x>y$. Hence $u$ is a strictly increasing continuous bijection from $\R$ onto $\R$, with Lipschitz inverse, and since $|\gamma(x)-\gamma(y)|\ge|u(x)-u(y)|$ we obtain \textit{i}). Let $\varphi=v\circ u^{-1}$. For $s=u(x)$ and $s'=u(x')$,
		\[
		|\varphi(s)-\varphi(s')|=\Bigl|\int_{x'}^{x}v'\Bigr|\le\tan(\Theta/2)\Bigl|\int_{x'}^{x}u'\Bigr|=\tan(\Theta/2)\,|s-s'| ,
		\]
		and $\Gamma=e^{i\theta}\{u(x)+iv(x)\}=e^{i\theta}\{s+i\varphi(s)\}$, which proves \textit{ii}). For \textit{iii}), recall that the arclength of the graph is $\sqrt{1+\varphi'(s)^2}\,ds$. Since $u$ is bi-Lipschitz, we may change variables $s=u(x)$, which turns it into
		\[
		\sqrt{1+\varphi'(u(x))^2}\;u'(x)\,dx .
		\]
		By the chain rule, almost everywhere $\varphi'(u(x))=v'(x)/u'(x)$, and then
		\[
		\sqrt{1+\varphi'(u(x))^2}\;u'(x)\,=\,\sqrt{u'(x)^2+v'(x)^2}=|\gamma'(x)| .\qedhere
		\]
	\end{proof}
	
	\begin{lemma}\label{lem:angle}
		Let $\lambda\in\C\setminus[-1,1]$. The vectors $\lambda-1$ and $\lambda+1$ are non-zero and not opposite. Let $\omega(\lambda)\in[0,\pi)$ be the angle between them, that is, the angle under which the segment $[-1,1]$ is seen from $\lambda$. Then:
		\begin{enumerate}[itemsep=0.75em]
			\item[i\,\normalfont{)}] the segment $\{\lambda-a:a\in[-1,1]\}$ is contained in a closed sector of opening $\omega(\lambda)$;
			\item[ii\,\normalfont{)}] $\cos\omega(\lambda)=\dfrac{|\lambda|^2-1}{|\lambda-1||\lambda+1|}$; in particular, $\omega(\lambda)<\pi/2$ if and only if $|\lambda|>1$;
			\item[iii\,\normalfont{)}] $\omega(\bar\lambda)=\omega(-\bar\lambda)=\omega(\lambda)$;
			\item[iv\,\normalfont{)}] if $\operatorname{Im}\lambda\ge0$, then $\omega(\lambda)=\arg(\lambda-1)-\arg(\lambda+1)$, with both arguments in $[0,\pi]$.
		\end{enumerate}
	\end{lemma}
	\begin{proof}
		If $\lambda-1$ and $\lambda+1$ were opposite, $\lambda+1=-c(\lambda-1)$ with $c>0$, solving gives $\lambda=\frac{c-1}{c+1}\in(-1,1)$, which is impossible. So the vectors are not opposite, and they are clearly non-zero.
		
		Regarding \textit{i}), observe that the closed sector generated by two vectors forming an angle less than $\pi$ is convex and has opening equal to that angle. Moreover $$\lambda-a=\frac{1+a}{2}(\lambda-1)+\frac{1-a}{2}(\lambda+1)$$ is a convex combination of the two, so \textit{i}) follows. 
		
		For \textit{ii}), recall that if $\alpha$ denotes the angle between two complex numbers $z_1,z_2$, then
		\[
		\cos{\alpha} = \frac{\operatorname{Re}(z_1\overline{z_2})}{|z_1||z_2|}.
		\]
		Therefore, since $\operatorname{Re}\bigl((\lambda-1)\overline{(\lambda+1)}\bigr)=|\lambda|^2-1$, we are done. 
		
		Note also that the maps $\lambda\mapsto\bar\lambda$ and $\lambda\mapsto-\bar\lambda$ are isometries that leave the set $[-1,1]$ invariant (the second one swaps its endpoints), so they preserve the angle under which the segment is seen and thus \textit{iii}) follows. 
		
		Finally, if $\operatorname{Im}\lambda>0$, both vectors lie in the open upper half-plane. Since $\lambda-1=(\lambda+1)-2$, the vector $\lambda-1$ has the larger argument, and the angle between them is the difference of the arguments. If \(\lambda\) is real and \(|\lambda|>1\), both arguments are equal, and the identity follows.
	\end{proof}

\section{Cauchy kernels and $L^2$ estimates}\label{sec:family}

We now turn to the Cauchy kernels $K_\lambda$ and the bilinear forms $r_\lambda$ introduced in the outline. The geometric results of the preceding section will allow us to obtain $L^2$ estimates using the Coifman--McIntosh--Meyer theorem.

For $\lambda\in\C\setminus[-1,1]$, the kernel
$K_\lambda$ is antisymmetric and satisfies
\begin{equation}\label{eq:boundK}
	|K_\lambda(x,y)|
	\le\frac{1}{d_\lambda|x-y|},
	\qquad x\ne y,
\end{equation}
where $d_\lambda=\dist(\lambda,[-1,1])$.



%
\begin{proposition}\label{prop:rlambda}
	For every $\lambda\in\C\setminus[-1,1]$ and
	$f,g\in\Cc$,
	\[
	|r_\lambda(f,g)|
	\le\frac{C_{M(\lambda)}}{d_\lambda}
	\norm f_2\norm g_2.
	\]
\end{proposition}
	
	\begin{proof}
		Write $\gamma=\gamma_\lambda$ and
		$J(x)=|\gamma'(x)|=|\lambda-A'(x)|$.
		Since $A'(x)\in[-1,1]$ almost everywhere,
		we have $J(x)\ge d_\lambda$ almost everywhere.
		
		By Lemma~\ref{lem:angle}(i), the tangent vectors
		$\gamma'(x)$ lie in a closed sector of opening
		$\omega(\lambda)$. Lemma~\ref{lem:cone}, applied
		with $\Theta=\omega(\lambda)$ and $m=d_\lambda$,
		shows that $\Gamma:=\gamma(\R)$ is a rotated
		Lipschitz graph with constant at most $M(\lambda)$
		and that $\gamma$ is injective.
		
		Set
		\[
		\rho(x,y)=|\gamma(x)-\gamma(y)|.
		\]
		Since $\gamma$ is injective, $\rho$ is symmetric,
		vanishes on the diagonal, and is strictly positive
		for $x\ne y$. Moreover, since $\gamma$ is Lipschitz,
		\[
		\rho(x,y)\le (|\lambda|+1)|x-y|.
		\]
		Thus Lemma~\ref{lem:weak} gives
		\[
		r_\lambda(f,g)
		=
		\lim_{\varepsilon\to0^+}
		\iint_{\rho(x,y)>\varepsilon}
		\frac{f(y)g(x)}{\gamma(x)-\gamma(y)}
		\,dx\,dy.
		\]
		
		Define $F$ and $G$ on $\Gamma$, almost everywhere
		with respect to $\sigma$, by
		\[
		F(\gamma(x))=\frac{f(x)}{J(x)},
		\qquad
		G(\gamma(x))=\frac{g(x)}{J(x)}.
		\]
		By Lemma~\ref{lem:cone}(iii), the change of
		variables $z=\gamma(x)$ and $w=\gamma(y)$ yields
		\[\iint_{\rho(x,y)>\varepsilon}
			\frac{f(y)g(x)}{\gamma(x)-\gamma(y)}
			\,dx\,dy=
			\int_\Gamma G(z)
			\left(
			\int_{|z-w|>\varepsilon}
			\frac{F(w)}{z-w}\,d\sigma(w)
			\right)d\sigma(z)=
			\int_\Gamma G(z)\,
			\mathcal C_{\Gamma,\varepsilon}F(z)\,d\sigma(z).
		\]
		Fubini's theorem is justified by the truncation
		and the compact supports of $f$ and $g$.
		
		By Cauchy--Schwarz and Definition~\ref{def:CM},
		\[
		\left|
		\int_\Gamma G\,
		\mathcal C_{\Gamma,\varepsilon}F\,d\sigma
		\right|
		\le
		C_{M(\lambda)}
		\norm F_{L^2(\sigma)}
		\norm G_{L^2(\sigma)}.
		\]
		Moreover,
		\[
		\norm F_{L^2(\sigma)}^2
		=\int_\R\frac{|f(x)|^2}{J(x)}\,dx
		\le\frac{\norm f_2^2}{d_\lambda},
		\]
		and similarly
		\[
		\norm G_{L^2(\sigma)}^2
		\le\frac{\norm g_2^2}{d_\lambda}.
		\]
		Letting $\varepsilon\to0^+$ completes the proof.
	\end{proof}

	\section{Recovering the commutators by contour integration}\label{sec:extraction}
We now establish the contour integral representation of $t_n$ in terms of $r_\lambda$ announced in the outline of the proof.	
	\begin{lemma}\label{lem:coef}
		Let $\Gamma$ be a piecewise $C^1$ Jordan curve contained in $\C\setminus[-1,1]$, positively oriented, with $[-1,1]$ in its interior. Then, for $f,g\in\Cc$,
		\[
		t_n(f,g)=\frac{1}{2\pi i}\int_\Gamma\lambda^n\,r_\lambda(f,g)\,d\lambda .
		\]
	\end{lemma}
	\begin{proof}
		Let $d_\Gamma=\dist(\Gamma,[-1,1])$, which is positive by compactness. The function
		\[
		\Phi(\lambda,x,y)=\lambda^n\,\frac{f(y)g(x)-f(x)g(y)}{2(x-y)(\lambda-q(x,y))}
		\]
		satisfies $|\Phi|\le|\lambda|^n\,c_{f,g}(2d_\Gamma)^{-1}\mathbf 1_{S\times S}(x,y)$, which is integrable with respect to $|d\lambda|\,dx\,dy$ on $\Gamma\times\R^2$. Fubini allows us to interchange the integrals. For fixed $x\ne y$, the number $q=q(x,y)$ lies inside $\Gamma$, and the Cauchy integral formula gives 
		\[
		\frac{1}{2\pi i}\int_\Gamma\frac{\lambda^n}{\lambda-q}\,d\lambda=q^n.
		\]
		Therefore
		\[
		\frac1{2\pi i}\int_\Gamma\lambda^nr_\lambda(f,g)\,d\lambda=\frac12\iint\frac{q(x,y)^n}{x-y}\,(f(y)g(x)-f(x)g(y))\,dx\,dy=t_n(f,g).\qedhere
		\]
	\end{proof}
	 For a compact curve
	$\Gamma\subset\C\setminus[-1,1]$, set
	$$\displaystyle
	M_\Gamma:=\sup_{\lambda\in\Gamma}M(\lambda).
	$$ Combining Lemma~\ref{lem:coef}, Proposition~\ref{prop:rlambda},
	and the monotonicity of $C_M$, we obtain the following
	estimate.
	
	\begin{corollary}\label{cor:principle}
	Let $\Gamma$ be a piecewise $C^1$ Jordan curve contained in $\C\setminus[-1,1]$, positively oriented, with $[-1,1]$ in its interior. Then
		\[
		\norm{T_{A,n}}
		\le\frac{C_{M_\Gamma}}{2\pi}
		\int_\Gamma\frac{|\lambda|^n}{d_\lambda}
		\,|d\lambda|.
		\]
	\end{corollary}
	
	Hence, to obtain a bound independent of $n$, it suffices to construct, for each $n$, a contour $\Gamma$ such that both $M_{\Gamma}$ and
$\int_{\Gamma}\frac{|\lambda|^n}{d_\lambda}|d\lambda|$
are uniformly bounded in $n$.

	\section{Proof of Theorem \ref{thm:main}}\label{sec:contour}
	As a first attempt, let us apply Corollary~\ref{cor:principle}
	to a circular contour and see what it yields.
	
	Fix $n\ge1$ and consider
	$\Gamma=\{|\lambda|=\rho\}$, where $\rho=1+1/n$.
	Since $|\lambda|>1$, Lemma~\ref{lem:angle}(ii)
	gives $\omega(\lambda)<\pi/2$. Thus
	$M(\lambda)\le1$ and $M_\Gamma\le1$.
	
	Since $[-1,1]$ is contained in the closed unit disc,
	\[
	d_\lambda\ge\rho-1=\frac1n.
	\]
	Moreover, as $[-1,1]\subset\R$, writing
	$\lambda=\rho e^{i\theta}$ gives
	\[
	d_\lambda\ge|\operatorname{Im}\lambda|
	=\rho|\sin\theta|\ge|\sin\theta|.
	\]
	Using $\sin\theta\ge2\theta/\pi$ on $[0,\pi/2]$
	and symmetry, we obtain
	\[
	\begin{aligned}
		\int_0^{2\pi}
		\frac{d\theta}{\max(1/n,|\sin\theta|)}
		&\le
		4\left(
		\int_0^{\pi/(2n)}n\,d\theta
		+
		\int_{\pi/(2n)}^{\pi/2}
		\frac{\pi}{2\theta}\,d\theta
		\right)
		&=2\pi(1+\log n).
	\end{aligned}
	\]
	Since
	\[
	|\lambda|^n
	=\left(1+\frac1n\right)^n\le e,
	\qquad
	|d\lambda|=\rho\,d\theta\le2\,d\theta,
	\]
	we conclude that
	\[
	\int_\Gamma
	\frac{|\lambda|^n}{d_\lambda}\,|d\lambda|
	\le4\pi e(1+\log n).
	\]
	Corollary~\ref{cor:principle} therefore yields
	\[
	\norm{T_{A,n}}\le2e\,C_1(1+\log n).
	\]
	
	This already improves the linear bound obtained
	in \cite{HMP}, but does not provide an estimate
	independent of $n$.
	
	On the circular contour, $|\lambda|^n$ does not
	decay, and the integral of $1/d_\lambda$ near
	$\pm1$ produces the logarithmic loss.
	To eliminate it, we need contours that enter
	the unit disc, where $|\lambda|^n$ decays.
	However, if $\lambda$ approaches $[-1,1]$
	tangentially, then $\omega(\lambda)$ may approach
	$\pi$, causing $M(\lambda)$ to become unbounded.
	The idea is therefore to approach the endpoints
	$\pm1$ at a fixed angle to the real axis. Along such segments, the distance to $[-1,1]$ is comparable to the distance $s$ from the endpoint, while $|\lambda|^n$ decays exponentially in $ns$. This decay compensates for the factor $1/d_\lambda$ in the contour integral.

	For convenience, we fix $\varphi=\arctan(1/2)\in(0,\pi/4)$ so that
	\[
	\sin{\varphi} = \frac{1}{\sqrt{5}}, \qquad \cos\varphi=\frac{2}{\sqrt5}, \qquad \tan{\varphi} = \frac{1}{2}.
	\]
	This way, 
	\[
	 M_0:=\cot\frac\varphi2 = \frac{1+\cos{\varphi}}{\sin{\varphi}} = 2+\sqrt{5}.
	\]
	For $0<\delta\le\frac12$ we set $s_\delta=\frac{2\delta}{\sqrt5}$.
	
	\begin{definition}[The contour]\label{def:contour}
		For $0<\delta\le1/2$, let $\Gamma_\delta$ be the positively oriented closed curve obtained by joining the following six pieces in the indicated order:
		\begin{enumerate}[itemsep=0.5em]
			\item right arc $\mathcal A_R$: $\lambda=1+\delta e^{it}$, with $t$ from $-(\pi-\varphi)$ to $\pi-\varphi$;
			\item side $\Sigma_R^+$: $\lambda=1-s+\frac i2 s$, with $s$ from $s_\delta$ to $1$, going from $p_R^+$ to $\frac i2$;
			\item side $\Sigma_L^+$: $\lambda=-1+s+\frac i2 s$, with $s$ from $1$ to $s_\delta$, going from $\frac i2$ to $p_L^+$;
			\item left arc $\mathcal A_L$: $\lambda=-1+\delta e^{it}$, with $t$ from $\varphi$ to $2\pi-\varphi$;
			\item side $\Sigma_L^-$: $\lambda=-1+s-\frac i2 s$, with $s$ from $s_\delta$ to $1$, going from $p_L^-$ to $-\frac i2$;
			\item side $\Sigma_R^-$: $\lambda=1-s-\frac i2 s$, with $s$ from $1$ to $s_\delta$, going from $-\frac i2$ to $p_R^-$;
		\end{enumerate}
		where $p_R^\pm=1-s_\delta\pm\frac i2s_\delta$ and $p_L^\pm=-1+s_\delta\pm\frac i2 s_\delta$.
	\end{definition}
	
The pieces join continuously. For instance,
$\displaystyle
1+\delta e^{i(\pi-\varphi)}
=1-\frac{2\delta}{\sqrt5}
+i\frac{\delta}{\sqrt5}
=p_R^+.
$
Similarly, $-1+\delta e^{i\varphi}=p_L^+$, while the two upper sides meet at $i/2$. The remaining junctions follow by conjugation. The resulting contour is illustrated in Figure~\ref{fig:contour}.
	\begin{figure}[H]
    \centering
    \begin{tikzpicture}[scale=2.7,>=stealth]
			\draw[->,gray] (-1.55,0) -- (1.6,0) node[right]{$\operatorname{Re}\lambda$};
			\draw[->,gray] (0,-0.75) -- (0,0.8) node[above]{$\operatorname{Im}\lambda$};
			\draw[gray!60,dotted] (0,0) circle (1);
			\draw[red!60,dashed] (1,0) arc[start angle=63.435,end angle=116.565,radius=2.2361];
			\draw[red!60,dashed] (1,0) arc[start angle=-63.435,end angle=-116.565,radius=2.2361];
			\draw[very thick] (-1,0) -- (1,0);
			\fill (1,0) circle (0.018) node[below right=-1pt]{\small $1$};
			\fill (-1,0) circle (0.018) node[below left=-1pt]{\small $-1$};
			\draw[thick,blue!70!black,
			postaction={decorate,decoration={markings,
					mark=at position 0.07 with {\arrow{>}},
					mark=at position 0.30 with {\arrow{>}},
					mark=at position 0.57 with {\arrow{>}},
					mark=at position 0.80 with {\arrow{>}}}}]
			(0.7764,-0.1118) arc[start angle=-153.435,end angle=153.435,radius=0.25]
			-- (0,0.5) -- (-0.7764,0.1118)
			arc[start angle=26.565,end angle=333.435,radius=0.25]
			-- (0,-0.5) -- cycle;
			\fill (0,0.5) circle (0.015) node[above right]{\small $\tfrac i2$};
			\fill (0,-0.5) circle (0.015) node[below right]{\small $-\tfrac i2$};
			\draw[<->] (1,0) -- ($(1,0)+(60:0.25)$) node[midway,above left=-2pt]{\scriptsize $\delta$};
			\node[blue!70!black] at (0.55,0.45) {$\Gamma_\delta$};
			\draw[red!60,->] (-0.75,0.55) node[above,font=\scriptsize]{$\omega(\lambda)=\pi-\varphi$} -- (-0.4,0.21);
		\end{tikzpicture}
		\caption{The contour $\Gamma_\delta$. The straight sides lie inside the unit disc (dotted), and their supporting lines are tangent at $\pm1$ to the boundary of the lens (dashed), where $\omega(\lambda)=\pi-\varphi$. Inside the lens, $\omega(\lambda)>\pi-\varphi$.}
		\label{fig:contour}
	\end{figure}
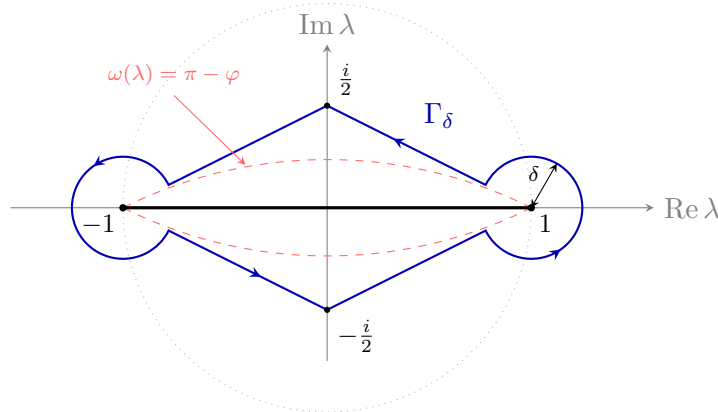	
\begin{lemma}\label{lem:jordan}
	$\Gamma_\delta$ is a positively oriented Jordan curve
	contained in $\C\setminus[-1,1]$, with $[-1,1]$
	in its interior.
\end{lemma}

\begin{proof}
	Let $\Gamma^+$ denote the upper half of
	$\Gamma_\delta$, traversed from $1+\delta$
	to $-1-\delta$. Along $\Gamma^+$, the real
	part is strictly decreasing: on the arcs
	this follows from the monotonicity of
	$\cos t$ on the relevant intervals, while
	on the sides it follows directly from their
	parametrizations.
	
	Hence $\Gamma^+$ is the graph of a continuous
	function $h$ on $[-1-\delta,1+\delta]$,
	positive in the interior and vanishing at
	the endpoints. The lower half is the graph
	of $-h$. Thus $\Gamma_\delta$ is a Jordan
	curve bounding the domain
	\[
	\{a+ib:|a|<1+\delta,\ |b|<h(a)\},
	\]
	which contains $[-1,1]$. Its positively oriented since $\Gamma^+$ is traversed
	from right to left.
\end{proof}
	
\begin{lemma}\label{lem:angular}
	For every $\delta\in(0,1/2]$ and every
	$\lambda\in\Gamma_\delta$,
	$
	\omega(\lambda)\le\pi-\varphi.
	$
	Consequently, $M_{\Gamma_\delta}\le M_0$.
\end{lemma}

\begin{proof}
	By the symmetries in Lemma~\ref{lem:angle}(iii),
	it suffices to consider $\Sigma_R^+$ and the
	upper half of $\mathcal A_R$.
	
	\emph{Side $\Sigma_R^+$.}
	For $\lambda=1-s+\frac{i}{2}s$, with
	$s\in[s_\delta,1]$, we have
	\[
	\arg(\lambda-1)=\pi-\varphi,
	\qquad
	\arg(\lambda+1)
	=\arctan\frac{s}{2(2-s)}\in[0,\varphi].
	\]
	Thus, by Lemma~\ref{lem:angle}(iv),
	$
	\omega(\lambda)
	=\arg(\lambda-1)-\arg(\lambda+1)
	\le\pi-\varphi.
	$
	
	\emph{Upper half of $\mathcal A_R$.}
	For $\lambda=1+\delta e^{it}$, with
	$t\in[0,\pi-\varphi]$, we have
	$
	\arg(\lambda-1)=t.
	$
	Since the sector of arguments in $[0,t]$
	is convex and contains both $2$ and
	$\delta e^{it}$, it also contains their
	sum $\lambda+1=2+\delta e^{it}$.
	Hence $\arg(\lambda+1)\in[0,t]$, and
	Lemma~\ref{lem:angle}(iv) gives
	$
	\omega(\lambda)
	=t-\arg(\lambda+1)
	\le t\le\pi-\varphi.
	$
	
	Finally, since
	$M(\lambda)=\tan(\omega(\lambda)/2)$
	and the tangent function is increasing on $[0,\pi/2)$,
	we obtain $M_{\Gamma_\delta}\le M_0$.
\end{proof}
	
	\begin{lemma}\label{lem:integral}
		There exists an absolute constant $K_0>0$ such that,
		for every integer $n\ge1$,
		\[
		\int_{\Gamma_{1/(n+1)}}
		\frac{|\lambda|^n}{d_\lambda}\,|d\lambda|
		\le K_0.
		\]
	\end{lemma}
	\begin{proof} Fix $n\geq 1$ and $\delta=1/(n+1)$. We estimate separately the contributions of the two arcs and the four straight sides.
		
		\emph{Arcs.} Let $\lambda=1+\delta e^{it}\in\mathcal A_R$. If $\operatorname{Re}\lambda\ge1$, the closest point of $[-1,1]$ is $1$, and $d_\lambda=\delta$. If $\operatorname{Re}\lambda<1$, then $\cos t<0$, that is, $|t|\in(\pi/2,\pi-\varphi]$, and $\operatorname{Re}\lambda\ge1-\delta\ge\frac12$. In that case the closest point is the real projection, and
		\[
		d_\lambda=\delta|\sin t|\ge\delta\sin(\pi-\varphi)=\frac{\delta}{\sqrt5},
		\]
		because $\sin$ is decreasing on $[\pi/2,\pi]$. In both cases $d_\lambda\ge\delta/\sqrt5$. Moreover $|\lambda|^n\le(1+\delta)^n\le e^{n/(n+1)}<e$, and the length of the arc is $2(\pi-\varphi)\delta<2\pi\delta$. The contribution of $\mathcal A_R$ is therefore less than $2\pi\delta\cdot e\sqrt5/\delta=2\pi\sqrt5\,e$. The same holds for $\mathcal A_L$, by the symmetry $\lambda\mapsto-\bar\lambda$, which preserves $|\lambda|$ and $d_\lambda$.
		
		\emph{Sides.} On $\Sigma_R^+$, $\operatorname{Re}\lambda=1-s\in[0,1)$, so $d_\lambda=|\operatorname{Im}\lambda|=\frac s2$, and $|d\lambda|=|-1+\frac i2|\,ds=\frac{\sqrt5}2\,ds$. Moreover, for $0\le s\le1$,
		\[
		|\lambda|^2=(1-s)^2+\frac{s^2}{4}=1-2s+\frac54s^2\le1-\frac34s ,
		\]
		because $\frac54s^2\le\frac54s$. Hence $|\lambda|^n\le(1-\frac34s)^{n/2}\le e^{-3ns/8}$. The contribution of $\Sigma_R^+$ is therefore at most
		\[
		\int_{s_\delta}^1e^{-3ns/8}\,\frac{2}{s}\,\frac{\sqrt5}{2}\,ds=\sqrt5\int_{ns_\delta}^{n}\frac{e^{-3u/8}}{u}\,du\le\sqrt5\int_{1/\sqrt5}^\infty\frac{e^{-3u/8}}{u}\,du<\infty,
		\]
		where we used $ns_\delta=\frac{2n}{\sqrt5(n+1)}\ge\frac{1}{\sqrt5}$ for $n\ge1$. The other three sides have the same $|\lambda|$ and the same $d_\lambda$ as functions of $s$: for $\Sigma_L^\pm$, $|\lambda|^2=(1-s)^2+\frac{s^2}4$ and $\operatorname{Re}\lambda=-1+s\in(-1,0]$. Hence they satisfy the same bound. Adding the six contributions gives the desired absolute constant $K_0$. 
	\end{proof}
	
 We are now ready to complete the proof of Theorem~\ref{thm:main}.
 \begin{proof}[Proof of Theorem~\ref{thm:main}]
 	Assume first that $\norm{A'}_\infty\le1$.
 	Fix $n\ge1$ and set $\delta=1/(n+1)$.
 	By Lemma~\ref{lem:jordan}, the contour $\Gamma_\delta$
 	satisfies the hypotheses of Corollary~\ref{cor:principle}.
 	Lemmas~\ref{lem:angular} and~\ref{lem:integral}
 	therefore yield
 	\[
 	\norm{T_{A,n}}
 	\le\frac{C_{M_0}}{2\pi}
 	\int_{\Gamma_\delta}
 	\frac{|\lambda|^n}{d_\lambda}\,|d\lambda|
 	\le\frac{C_{2+\sqrt5}K_0}{2\pi}.
 	\]
 	This constant is finite by Theorem~\ref{thm:CMM}
 	and independent of $n$.
 	
 	For general $A$, let $L=\norm{A'}_\infty$.
 	If $L>0$, applying the preceding estimate to $A/L$
 	and using $T_{A,n}=L^nT_{A/L,n}$ gives the result.
 	If $L=0$, then $T_{A,n}=0$.
 \end{proof}

	\section{Sharpness and remarks}\label{sec:remarks}
	
	\begin{remark}
		Let $A(x)=Lx+b$ with $L>0$. All secant slopes equal $L$, and therefore $T_{A,n}=L^nH_0$, with $H_0f=\pv\int\frac{f(y)}{x-y}\,dy$. With the convention $\hat f(\xi)=\int e^{-ix\xi}f(x)\,dx$, the multiplier of $H_0$ is $-i\pi\operatorname{sgn}\xi$, so $\norm{T_{A,n}}=\pi L^n$. Thus the dependence on $\|A'\|_\infty^n$ is sharp. 
	\end{remark}
	
	\begin{remark}
		Only the finiteness of $C_{M_0}$ for the fixed value $M_0=2+\sqrt5$ is used. No quantitative bound on the dependence of $C_M$ on $M$ is needed. The value $\varphi=\arctan\frac12$ is not special. Any angle $\varphi\in(0,\pi/4)$  yields a similar construction, with straight sides of slopes $\pm\tan\varphi$. The resulting bounds depend on $\varphi$, but remain uniform in $n$.
	\end{remark}
	
	  \begin{remark}
		The theorem concerns the principal-value operators
		$T_{A,n}$, represented by the bilinear forms $t_n$.
		We do not establish bounds uniform in $n$ for the
		truncated operators or their maximal truncations.
	\end{remark}
	
	\medskip
\noindent\textbf{AI disclosure.}
The authors thank Juan Carlos Cantero for providing us with access
to OpenAI's GPT-6 model, which was used for mathematical
discussions and suggestions concerning the proof.
The authors take full responsibility for the mathematical
arguments and conclusions.
	

	\vspace{1.5cm}
	{\small
		\begin{tabular}{@{}l}
			\textsc{Joan\ Hernández,} \\ \textsc{Departament de Matem\`{a}tiques, Universitat Aut\`{o}noma de Barcelona,}\\
			\textsc{08193, Bellaterra (Barcelona), Catalonia.}\\
			{\it E-mail address}\,: \href{mailto:joan.hernandez@uab.cat}{\tt{joan.hernandez@uab.cat}}
		\end{tabular}
	}
	
	{\small
		\begin{tabular}{@{}l}
			\textsc{Joan\ Mateu,} \\ \textsc{Departament de Matem\`{a}tiques, Universitat Aut\`{o}noma de Barcelona,}\\
			\textsc{08193, Bellaterra (Barcelona), Catalonia.}\\
			{\it E-mail address}\,: \href{mailto:joan.mateu@uab.cat}{\tt{joan.mateu@uab.cat}}
		\end{tabular}
	}

	{\small
		\begin{tabular}{@{}l}
			\textsc{Laura\ Prat,} \\ \textsc{Departament de Matem\`{a}tiques, Universitat Aut\`{o}noma de Barcelona,}\\
			\textsc{08193, Bellaterra (Barcelona), Catalonia.}\\
			{\it E-mail address}\,: \href{mailto:laura.prat@uab.cat}{\tt{laura.prat@uab.cat}}
		\end{tabular}
	}  
	
\end{document}